\documentclass[12pt]{article}
\usepackage{fullpage,amsthm,amssymb,tikz,amsmath,verbatim,moresize}
\usepackage[font=footnotesize]{caption} 
\usepackage{moresize}
\def\vph{\varphi}
\def\wh{\widehat}

\def\ch{{\mathrm{ch}}}

\newtheorem{lem}{Lemma}

\newtheorem{thm}{Theorem}
\newtheorem{thmR}{Theorem}

\newtheorem*{mad-lem}{Mad Lemma}

\newtheorem{conj}{Conjecture}
\theoremstyle{definition}

\newtheorem{clm}{Claim}

\newcommand{\Kc}{\textrm{Kc}}
\newcommand\floor[1]{\lfloor#1\rfloor}
\newcommand\ceil[1]{\lceil#1\rceil}
\newcommand\Ix[1]{\mathbf{1}_#1}

\RequirePackage{marginnote,hyperref}

\newenvironment{clmproof}[1]{\par\noindent\underline{Proof.}\space#1}{\leavevmode\unskip\penalty9999\hbox{}\nobreak\hfill\quad\hbox{$\diamondsuit$}\smallskip}

\title{Kempe Classes and Almost Bipartite Graphs}
\author{Daniel W. Cranston\thanks{%
Department of Computer Science, Virginia Commonwealth
University, Richmond, VA, USA;
\texttt{dcranston@vcu.edu}
}
\and
Carl Feghali\thanks{%
Univ. Lyon, EnsL, UCBL, CNRS, LIP, F-69342, Lyon Cedex 07, France,
\texttt{feghali.carl@gmail.com}%
}
}

\begin{document}
\date{}
\maketitle
\abstract{
Let $G$ be a graph and $k$ be a positive integer, and let $\Kc(G, k)$ denote the number of Kempe equivalence classes for the $k$-colorings of $G$.   In 2006, Mohar noted that $\Kc(G, k) = 1$ if $G$ is bipartite. As a generalization, we show that $\Kc(G, k) = 1$ if $G$ is formed from a bipartite graph by adding any number of edges less than $\binom{\ceil{k/2}}2+\binom{\floor{k/2}}2$. We show that our result is tight (up to lower order terms)  by constructing, for each $k \geq 8$, a graph $G$ formed from a bipartite graph by adding $(k^2+8k-45+1)/4$ edges such that $\Kc(G, k) \geq 2$. This refutes a recent conjecture of Higashitani--Matsumoto.
}
\section{Introduction}
In a \emph{(proper) $k$-coloring} of a graph $G$, each vertex is assigned a color from
$\{1,\ldots,k\}$ and adjacent vertices receive distinct colors.  Given a
$k$-coloring $\vph$, an \emph{$i,j$-Kempe swap} interchanges the colors of some component of the subgraph
induced under $\vph$ by colors $i$ and $j$. Clearly, this
yields a new proper $k$-coloring.  

Kempe swaps were introduced in the late
1800s by Alfred Kempe, in a failed attempt to prove the 4 Color Theorem.
However, the idea was largely salvaged by Heawood, who used this technique to
prove the 5 Color Theorem.  Kempe swaps did play a crucial role in the eventual
proof of the 4 Color Theorem~\cite{AH} and its later
reproof~\cite{RSST}.  Kempe swaps have also played a central role in the study
of edge-coloring, including work of Vizing~\cite{vizing1,vizing2},
Kierstead~\cite{kierstead,EFK}, Tashkinov~\cite{tashkinov},
and many others~\cite{HiltonZhao4,GS-proved,GS-survey}.
All of this work asks one meta-question ``Does a given graph (or line graph)
admit a $k$-coloring?''

However, this sustained focus on the use of Kempe swaps eventually prompted two
interesting related questions: ``Given a graph $G$ and two $k$-colorings
$\vph_1$ and $\vph_2$, is it possible to transform $\vph_1$ into $\vph_2$ by a
sequence of Kempe swaps?'' and 
``Is it possible to transform \emph{every} $k$-coloring into every
other $k$-coloring by a sequence of Kempe swaps?''
For a graph $G$ and a positive integer $k$, two $k$-colorings are
\emph{equivalent} if one can be transformed to the other by a sequence of Kempe
swaps. Let $\Kc(G,k)$ denote the number of equivalence classes.
Now the latter question can be rephrased as asking whether or not $\Kc(G,k)=1$.
(Such questions fall, more broadly, under the heading of \emph{reconfiguration},
which is surveyed by van den Heuvel~\cite{vdH-survey} and
Nishimura~\cite{nishimura-survey}.)  

Early work in this area was done by
Vizing~\cite{vizing1,vizing2}
(for edge-coloring), Meyniel~\cite{meyniel}, Fisk~\cite{fisk}, and Las Vergnas and
Meyniel~\cite{LVM}.  
In 2006 this line of study was reinvigorated by a
beautiful survey of Mohar~\cite{mohar-survey}, which proposed a variety of open
questions.  The abundance of results that has followed includes work on regular
graphs~\cite{BBFJ,FJP}, planar graphs~\cite{feghali-critical,strengthened-meyniel}, 
grids~\cite{mohar-salas,salas-sokal,CM-toroidal}, line graphs
(edge-coloring)~\cite{MMS,BDKLN,narboni}, and much
more; it even has a list-coloring variant~\cite{list-mohars, edge-swap}.

Mohar \cite{mohar-survey} noted (and it is easy to check) that $\Kc(B,k)=1$ when $B$ is any
bipartite graph and $k\ge 2$.  This motivated Higashitani and Matsumoto \cite{HM} to study
$\Kc(B+E_{\ell},k)$, where $B + E_{\ell}$ is a graph formed from $B$ by adding $\ell$ edges.  Of course, every graph can be viewed as a $B+E_{\ell}$
graph for some $\ell$.  So it is natural to impose bounds on
$\ell$ in terms of some other features of the graph.  This led
Higashitani and Matsumoto to conjecture the following, which they proved in the
case $k=4$.

\begin{conj}[Higashitani--Matsumoto; now disproved for $k\ge 8$]
\label{main-conj}
If $G$ is a $(k-1)$-colorable $B + E_\ell$ graph with $k \ge 4$ and $\ell <
\binom{k}{2}$, then $\Kc(G, k) = 1$.
\end{conj}

The assumptions on $\ell$ and the chromatic number of $G$ in Conjecture \ref{main-conj} 
are necessary \cite{HM}.  Here we show that Conjecture~\ref{main-conj} is false for
all $k\ge 8$ (we suspect that it is true for $k\in\{5,6,7\}$). Let $\Ix{x}$ denote 
the function on positive integers $x$ that is 1 if $x$
is odd and 0 if $x$ is even.  That is, $\Ix{x}$ is equal to $x\bmod{2}$.

\begin{thm}\label{thm1}
For all $k\ge 8$, Conjecture~\ref{main-conj} is false.
Furthermore, for all $k\ge 12$, there exist $(k-1)$-colorable graphs $G_k$
with $\Kc(G_k,k)\ge 2$ and $G_k=B+E_{\ell}$ where $\ell=(k^2+8k-45+\Ix{{k-1}})/4$.
\label{thm1-first}
\end{thm}

In our next result, we show that, in a certain sense, our constructions are the worst
possible (up to lower order terms). 

\begin{thm}\label{thm2}
If $G=B+E_{\ell}$ with
$\ell<\binom{\ceil{k/2}}2+\binom{\floor{k/2}}2=\frac{k^2-2k+\Ix{k}}4$, then
$\Kc(G,k)=1$.
\end{thm}

\section{Proofs}



For easy reference, we record the following observation.

\begin{lem}
\label{f-lem}
If $f(x) := \binom{\ceil{x/2}}2+\binom{\floor{x/2}}2$,
then $f(x) = (x^2-2x+\Ix{x})/4$.
\end{lem}
\begin{proof}
If $x$ is even, then $f(x) = 
2\binom{x/2}2 =
\frac{x}2\frac{x-2}2 = \frac{x^2-2x}4=\frac{x^2-2x+\Ix{x}}4$.
If $x$ is odd, then $f(x)=
\binom{(x+1)/2}2+\binom{(x-1)/2}2 =
\frac{x^2-1+x^2-4x+3}8 = \frac{x^2-2x+1}4=\frac{x^2-2x+\Ix{x}}4$.
\end{proof}

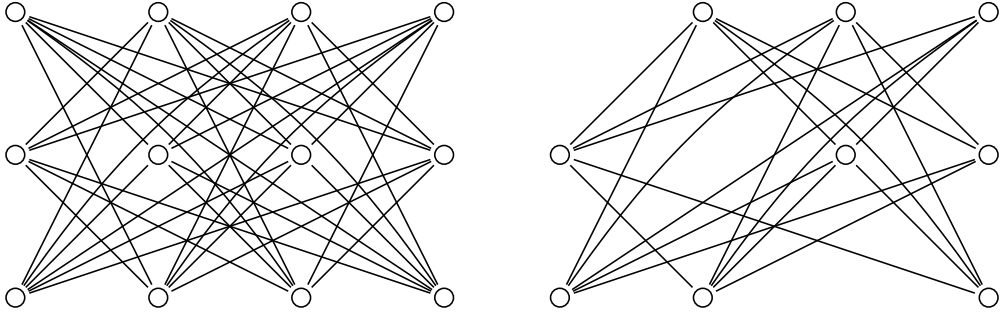
\begin{figure}[!h]
\centering
\begin{tikzpicture}[semithick, scale=1.9]
\tikzstyle{uCircle}=[shape = circle, minimum size = 5pt, inner sep = 2.5pt, outer sep = 2pt, fill = white, draw]
\tikzset{every node/.style=uCircle}

\foreach \i/\j in {1/1, 2/1, 3/1, 4/1, 1/2, 2/2, 3/2, 4/2, 1/3, 2/3, 3/3, 4/3}
\draw (\i,-\j) node (v\i\j) {};

\foreach \i/\j/\k/\l in {
1/1/2/2, 1/1/3/2, 1/1/4/2, 1/1/2/3, 1/1/4/3, 
1/2/2/1, 1/2/3/1, 1/2/4/1, 1/2/2/3, 1/2/3/3, 1/2/4/3,
1/3/2/1, 1/3/4/1, 1/3/2/2, 1/3/3/2, 1/3/4/2, 
2/1/3/2, 2/1/4/2, 2/1/3/3, 
2/2/3/1, 2/2/4/1, 2/2/3/3, 2/2/4/3,
2/3/3/1, 2/3/3/2, 2/3/4/2, 
3/1/4/2, 3/1/4/3,
3/2/4/1, 3/2/4/3,
3/3/4/1, 3/3/4/2
}
\draw (v\i\j) -- (v\k\l);
\draw (v11) edge[bend left=10] (v33);
\draw (v13) edge[bend left=10] (v31);
\draw (v21) edge[bend left=10] (v43);
\draw (v23) edge[bend left=10] (v41);

\begin{scope}[xshift=1.5in]

\foreach \i/\j in {2/1, 3/1, 4/1, 1/2, 3/2, 4/2, 1/3, 2/3, 4/3}
\draw (\i,-\j) node (v\i\j) {};

\foreach \i/\j/\k/\l in {
%
1/2/2/1, 1/2/3/1, 1/2/4/1, 1/2/2/3, 1/2/4/3, 
1/3/2/1, 1/3/4/1, 1/3/3/2, 1/3/4/2, 
2/1/3/2, 2/1/4/2, 
%
%
2/3/3/1, 2/3/3/2, 2/3/4/2, 
3/1/4/2, 3/1/4/3,
3/2/4/1, 3/2/4/3
%
}
\draw (v\i\j) -- (v\k\l);
\draw (v13) edge[bend left=10] (v31);
\draw (v21) edge[bend left=10] (v43);
\draw (v23) edge[bend left=10] (v41);

\end{scope}
\end{tikzpicture}
\caption{Left: The graph $K_3\times K_4$, also called $H$. Right: The graph
$\widehat{H}$ formed from $K_3\times K_4$ by deleting vertices $(1,1)$, $(2,2)$,
and $(3,3)$. We note that $\chi(\widehat{H})=3$ but $\Kc(\widehat{H},4)\ge
2$.\label{fig0}}
\end{figure}

The proof of Theorem~\ref{thm1} relies crucially on a graph of order~9, shown on
the right in Figure~\ref{fig0}.

\begin{lem}\label{lem:h}
There exists a graph $\wh{H}$ of order 9 such that
$\Kc(\wh{H},\chi(\wh{H})+1)\ge 2$.
\end{lem}
\begin{proof}
Let $H:=K_3\times K_4$, the categorical product of $K_3$ and $K_4$, whose vertices are pairs
$(i,j)$, where $i\in\{1,2,3\}$ and $j\in\{1,2,3,4\}$ and $(i_1,j_1)$ is adjacent to
$(i_2,j_2)$ if and only if $i_1\ne i_2$ and $j_1\ne j_2$.  Form $\wh{H}$ from
$H$ by deleting $(1,1)$, $(2,2)$, and $(3,3)$.  Clearly
$\chi(H)=\chi(\wh{H})=3$.  Let $\alpha$ (resp.~$\beta$) be the coloring of $H$
that gives each vertex the color equal to its first (resp. second) coordinate.
Note that $\beta$ is frozen (up to permuting colors), since the subgraph induced
by any 2 color classes is connected.  Since $\alpha$ has different color classes
than $\beta$, we conclude that $\Kc(H,\chi(H)+1)\ge 2$.

 Let 
$\wh{\beta}$ denote the restriction of $\beta$ to $\wh{H}$. Now $\wh{\beta}$ is also frozen under Kempe swaps
(up to permuting colors); this is straightforward to check.  Thus,
$\Kc(\wh{H},\chi(\wh{H})+1)\ge 2$.
\end{proof}

We will see later that every graph $\wh{H}$ such that $\Kc(\wh{H},\ch(\wh{H})+1)\ge 2$ has
order at least 9. 
(We provide the details in Theorem~\ref{smallest-lem}.)
This fact is of interest because the graph's order determines the linear term
for $\ell$ in our construction.

We now prove Theorem~\ref{thm1}. 
For easy reference, we restate it.

\begin{thmR}
For all $k\ge 8$, Conjecture~\ref{main-conj} is false.
Furthermore, for all $k\ge 12$, there exist $(k-1)$-colorable graphs $G_k$
with $\Kc(G_k,k)\ge 2$ and $G_k=B+E_{\ell}$ where $\ell=(k^2+8k-45+\Ix{{k-1}})/4$.
\end{thmR}

\begin{proof}
Let $J$ be an arbitrary graph and form $J'$ from $J$ by adding a dominating
vertex $v$.  It is easy to check that $\Kc(J,\chi(J)+1)=1$ if and only if
$\Kc(J',\chi(J')+1)=1$.  Moreover,
since $v$ is dominating, every subgraph induced by $v$ and some other color
class is connected.  Therefore, if $J$ has a frozen coloring, then so does $J'$. 

Let $\wh{H}$ be as in Lemma~\ref{lem:h} and form graph $G_k$ from the complete bipartite graph $K_{9,k-4}$
by adding the edges of $\wh{H}$ on the part of size 9 and adding all $\binom{k-4}2$ edges on the other part.
Since $\Kc(\wh{H},4)\ge 2$, also $\Kc(G_k,k)\ge 2$ for all $k\ge 4$, by 
the previous paragraph, inducting on $k$.
Now it suffices to choose $k$ so $G_k$ is an almost bipartite graph
$B+E_{\ell}$ where $\ell<\binom{k}{2}$.  Note that $\wh{H}$ is a 9-vertex graph
with 21 edges.
The number of edges within the part of size $k-4$ is simply $\binom{k-4}2$.
So we need $\binom{k-4}2+21 < \binom{k}2$.
This inequality expands to $k^2-9k+20+42 < k^2-k$, which holds for all $k\ge
8$.

For $k\ge 12$, to minimize the number of edges within parts, we split the
vertices of $G_k$ equally between parts, with all vertices of $\wh{H}$ in the same
part.  
Define $f(x)$ as in Lemma~\ref{f-lem}.
The number of edges within parts is now
$f(k+5)-(\binom{9}2-21)=\frac{(k+5)^2-2(k+5)+\Ix{{k-1}}}4-15 =
\frac{k^2+10k+25-2k-10+\Ix{{k-1}}}4-15 = \frac{k^2+8k-45+\Ix{{k-1}}}4$.
\end{proof}

We now prove Theorem \ref{thm2} which shows that the leading term, $k^2/4$, in the construction of 
Theorem~\ref{thm1} is best possible.
For easy reference, we restate it.

\begin{thmR}
If $G=B+E_{\ell}$ with $\ell<\binom{\ceil{k/2}}2 +
\binom{\floor{k/2}}2=\frac{k^2-2k+\Ix{k}}4$, then $\Kc(G,k)=1$.
\end{thmR}

\begin{proof}
Suppose the theorem is false, and let $G$ be a counterexample with $|V(G)|$ as
small as possible. 
Now $|V(G)| >  k$ since otherwise we can simply recolor all vertices with
distinct colors and permute colors as desired.  We require two claims.

\begin{clm}\label{claim:1}
Every $k$-coloring of $G$ is $k$-equivalent to some $(k-1)$-coloring.
\end{clm}
\begin{clmproof}
Fix an arbitrary $k$-coloring $\varphi$ of $G$ that uses exactly $k$ colors.  
An $i,j$-edge for $\vph$ is an edge of $E_{\ell}$ with colors $i$ and $j$ on its endpoints.

We first show that if there exists a pair of colors $i,j$ such that $E_{\ell}$ 
does not have any $i,j$-edge, but color $i$ or $j$ (or both) appears
on both parts $S$ and $T$, then there exists a sequence of Kempe swaps such
that after performing them colors $i$ and $j$ each appear on at most one part.  
(Specifically, color $i$ will appear only on $S$ and color $j$ only on $T$.)  
To do this, whenever $i$ appears on $T$ or $j$ appears on $S$, we use an $i,j$-Kempe 
swap.  Because $E_{\ell}$ contains no $i,j$-edge, this process achieves the desired
property.  

Each time that we perform the steps in the previous paragraph, the number
of colors used on at most one part increases, so this process eventually halts.
%
Thus, we can assume from now on, that for every pair of colors $i, j$ of
$\varphi$, if $E_{\ell}$ has no $i,j$-edge, then in $\vph$ colors $i$ and $j$
each appear on at most one part.  

Suppose colors $i$ and $j$ each appear only on the same one part, and that
part has no $i,j$-edge.  Now colors $i$ and $j$ can be ``merged'' by recoloring
each vertex colored $i$ with $j$.  We call $i,j$ a mergeable pair for $\vph$.
To prove the claim, it suffices to show that $\vph$ must have a mergeable pair,
which we do now.

By hypothesis,
we have 
$\ell\le\frac{k^2-2k+\Ix{k}}4-1 =
\frac{k^2-2k-4+\Ix{k}}4$.  Thus, the number of pairs $i,j$ with no $i,j$-edge in
$E_{\ell}$ is at least $\binom{k}2-\ell\ge\frac{k(k-1)}2-\frac{k^2-2k-4+\Ix{k}}4 =
\frac{2k^2-2k-(k^2-2k-4+\Ix{k})}4 = \frac{k^2+4-\Ix{k}}4$.   If such a pair is 
not mergeable, then colors $i$ and $j$ must be on opposite parts of $B$.
We form a bipartite graph $B'$ from $B$ as follows.  For each color $i$ such
that there exists a color $j$ for which $E_{\ell}$ has no $i,j$-edge, we
identify all vertices colored $i$ under $\vph$ (as shown above, all
of these are in the same part).  For each $i$ such that
$E_{\ell}$ has an $i,j$-edge for all $j\ne i$, we delete all vertices colored
$i$.  The resulting bipartite graph $B'$ has at most $k$ vertices, so has at most
$\frac{k^2-\Ix{k}}4$ edges.  
Thus, the number of mergable pairs is at least
$\frac{k^2+4-\Ix{k}}4-\frac{k^2-\Ix{k}}4=1$.
\end{clmproof}

\begin{clm}
$G[S]$ and $G[T]$ are both complete graphs.
\end{clm}
\begin{clmproof}
Assume instead that there exist $x,y\in S$ with $xy\notin
E_{\ell}$.  Let $\alpha$ and $\beta$ be arbitrary $(k-1)$-colorings of $G$,
each avoiding color $k$.  By Claim \ref{claim:1}, every $k$-coloring of $G$ is 
$k$-equivalent to a $(k-1)$-coloring of $G$. So it suffices to show
$\alpha$ and $\beta$ are $k$-equivalent.

In each of $\alpha$ and $\beta$, recolor vertices $x$ and $y$ with color $k$ and
identify $x$ and $y$.  Call the resulting graph $G'$ and the resulting colorings
$\alpha'$ and $\beta'$.  By minimality, $\Kc(G',k)=1$.
In particular, $\alpha'$ and $\beta'$ are $k$-equivalent.  Thus, 
$\alpha$ and $\beta$ are $k$-equivalent in $G$.  
Hence, $\Kc(G,k)=1$, and $G$ is not a counterexample, a contradiction.
\end{clmproof}

\medskip

Since $G[S]$ and $G[T]$ are both complete, 
$\ell = \binom{|S|}2+\binom{|T|}2$.  By the convexity of $\binom{x}2$, we have
$\ell\ge\binom{\ceil{(|S|+|T|)/2}}2+\binom{\floor{(|S|+|T|)/2}}2$.
By hypothesis, 
$\ell<\binom{\ceil{k/2}}2+\binom{\floor{k/2}}2$.
Thus, $|V(G)|=|S|+|T|<k$.  This contradicts the second sentence of the proof,
and this final contradiction proves the theorem. 
\end{proof}


We end this note with the proof of the following aforementioned result. 

\begin{thm}
\label{smallest-lem}
If $\Kc(G,\chi(G)+1)\ge 2$, then $|V(G)|\ge 9$.
\end{thm}
\begin{proof}
We prove that if $|V(G)|\le 8$, then $\Kc(G,\chi(G)+1)=1$.
We proceed by induction on $|V(G)|$.  The statement is trivial if $G$ or
its complement $\overline{G}$ is a complete graph, which includes the base case 
$|V(G)|\le 2$.
Mohar \cite{mohar-survey} observed  that if $\chi(G)=2$, then $\Kc(G,k)=1$
for all $k\ge 2$.  So we assume henceforth that $\chi(G)\ge 3$.  If there exists
$v$ such that $d(v)<\chi(G)+1$, then $\Kc(G,\chi(G)+1)=1$ if (and only if)
$\Kc(G-v,\chi(G)+1)=1$, so we are done by the induction hypothesis (this was
proved by Las Vergnas and Meyniel~\cite{LVM}; Mohar also includes a
proof~\cite[Prop~2.4]{mohar-survey}).  So we assume
henceforth that $\delta(G)\ge\chi(G)+1\ge 4$.

If $G$ has a dominating vertex $v$, then $\Kc(G,\chi(G)+1)=1$ if and only if
$\Kc(G-v,\chi(G))=\Kc(G-v,\chi(G-v)+1)=1$, so we are again done by the
induction hypothesis.  However, we prove something stronger: If there exists
an independent set $S\subset V(G)$ and there exists $v\in S$
such that $v$ is dominating in $G-S$, then we are also done by induction.
To see this, note that an arbitrary $(\chi(G)+1)$-coloring of $G$ is
$(\chi(G)+1)$-equivalent to one that uses a common color on all of $S$ (recolor
$S$ with the color of $v$).
But now deleting $S$ gives $\Kc(G,\chi(G)+1)=1$ if
$\Kc(G-S,\chi(G))=\Kc(G-S,\chi(G-S)+1)=1$, and the final statement holds by the
induction hypothesis.  In $\overline{G}$, the vertex $v$ above is a
simplicial vertex, i.e., its neighborhood is a clique.  Thus, we call $v$ an
\emph{antisimplicial vertex}.  Henceforth, we assume that $G$ has no
antisimplicial vertex.  This implies that $\delta(\overline{G})\ge 2$.
Thus, $|V(G)|\ge \delta(G)+1+\delta(\overline{G})\ge 4+1+2=7$.

If $\chi(G)\ge 4$, then $\delta(G)\ge 5$. Since $|V(G)|\le 8$, we get that
$|V(G)|=8$ and $\overline{G}$ is 2-regular.  Thus, $\overline{G}$ is (i)
$C_3+C_5$ or (ii) $C_4+C_4$ or (iii) $C_8$ (here $+$ denotes disjoint union). 
In (i), $G$ contains an antisimplicial vertex, a contradiction.  
(ii) Fix an arbitrary 4-coloring $\beta$.  
We will show that every 5-coloring of $G$ is 5-equivalent to $\beta$.
Starting from an arbitrary 5-coloring $\alpha$,
if any color appears on a single vertex, we recolor that vertex to merge two
color classes.  Now we can use Kempe swaps to arrange that every color that
appears in $\beta$ also appears in our current coloring on the same
$\overline{C_4}$ as in $\beta$.  Finally, we recolor
each $\overline{C_4}$ independently, since $\Kc(\overline{C_4},k)=1$ for all
$k\in\{2,3\}$.  

(iii) Each 5-coloring has either (a) four color classes of size
2 or (b) three color classes of size 2 and two of size 1.  In case (a), we have
two possibilities (up to permuting colors).  In case (b), the color classes of
size 1 are at distance 1 or 3 along the cycle in $\overline{G}$, and by
recoloring along a path between them (in $\overline{G}$), we can reach each of
the possibilities in (a).  Thus it suffices to note that to permute colors $i$
and $j$ in a 4-coloring of $\overline{C_8}$, we simply use an $i,j$-Kempe swap.

So we now assume that $\chi(G)=3$. 
If $|V(G)|=7$, then
$\overline{G}$ is 2-regular, so $\overline{G}$ is $C_3+C_4$ or $C_7$.  The first
case yields a contradiction, since $G$ contains an antisimplicial vertex.  And
the second yields a contradiction, since $\chi(G)=4$.  Thus, we henceforth
assume that $|V(G)|=8$.

To recap, we know that (i) $|V(G)|=8$, (ii) $\chi(G)=3$, (iii) $\delta(G)\ge 4$,
(iv) $\delta(\overline{G})\ge 2$, and (v) $G$ contains no antisimplicial vertex.
Note that (i) and (iii) imply $\Delta(\overline{G})\le 3$.  Together with (v),
this implies that $\omega(\overline{G})\le 3$. In fact, by (ii),
$\overline{G}$ contains a (spanning) copy of $K_3+K_3+K_2$; call it $S$. 
If all neighbors of a vertex $v$ in $\overline{G}$ are in $S$, then $v$ is
antisimplicial, a contradiction.  Thus, 
$|E(\overline{G})|\ge |E(G[S])|+|V(G)|/2=7+4=11$.
Further, we have only 3 possibilities for $\overline{G}$, shown in
Figure~\ref{fig1}.  Let $\alpha$ be the 3-coloring of $G$ that uses a common color
on each clique of $S$ (in $\overline{G}$).
In each case for $G$, we prove that every 4-coloring is 4-equivalent to
$\alpha$.  Fix an arbitrary 4-coloring of $G$; call it $\beta$.

\begin{figure}[!t]
\centering
\def\rad{2cm}
\begin{tikzpicture}[ultra thick, scale=.4, xscale=.8]
\tikzstyle{uCircle}=[shape = circle, minimum size = 6pt, inner sep = 3pt, outer sep = 0pt, fill = white, draw]
\tikzstyle{uF}=[shape = circle, minimum size = 6pt, inner sep = 1.3pt, outer sep = 0pt, fill = white, draw]
\tikzset{every node/.style=uCircle}
\draw (0,0) node[uF] (A) {\tiny{3}} --++ (150:\rad) node[uF] (B) {\tiny{1}} --++ (270:\rad) node[uF] (C) {\tiny{2}}
--++ (30:\rad) ++ (0:\rad) node[uF] (D) {\tiny{3}} --++ (0:\rad) node[uF] (E) {\tiny{4}} ++ (0:\rad)
node[uF] (F) {\tiny{4}} --++ (30:\rad) node[uF] (G) {\tiny{1}} --++ (270:\rad) node[uF] (H) {\tiny{2}} --++ (150:\rad);
\draw[semithick] (A) -- (D) (E) -- (F) (B) -- (G) (C) -- (H);

\draw (1.5*\rad,-1.1*\rad) node[draw=none, fill=none] {(i)};

\begin{scope}[xshift=6in]
\draw (0,0) node (A) {} --++ (150:\rad) node (B) {} --++ (270:\rad) node (C) {}
--++ (30:\rad) ++ (0:\rad) node (D) {} --++ (0:\rad) node (E) {} ++ (0:\rad)
node (F) {} --++ (30:\rad) node (G) {} --++ (270:\rad) node (H) {} --++ (150:\rad);
\draw[semithick] (A) -- (D) (E) -- (F) (C) -- (H);
\draw (B) edge[semithick,bend left=20] (E);
\draw (D) edge[semithick,bend left=20] (G);
\draw (1.5*\rad,-1.1*\rad) node[draw=none, fill=none] {(ii)};
\end{scope}

\begin{scope}[xshift=12in]
\draw (0,0) node (A) {} --++ (150:\rad) node (B) {} --++ (270:\rad) node (C) {}
--++ (30:\rad) ++ (0:\rad) node (D) {} --++ (0:\rad) node (E) {} ++ (0:\rad)
node (F) {} --++ (30:\rad) node (G) {} --++ (270:\rad) node (H) {} --++ (150:\rad);
\draw[semithick] (A) -- (D) (E) -- (F) (C) -- (H);
\draw (B) edge[semithick,bend left=15] (D);
\draw (E) edge[semithick,bend left=15] (G);
\draw (1.5*\rad,-1.1*\rad) node[draw=none, fill=none] {(iii)};
\end{scope}
\end{tikzpicture}
\caption{The three possibilities for $\overline{G}$ when $|V(G)|=8$.
The spanning subgraph $S$ is in bold. 
For Case (i), we show the only 4-coloring in which no edge of $S$
has the same color on both endpoints (up to permuting colors).
\label{fig1}}
\end{figure}
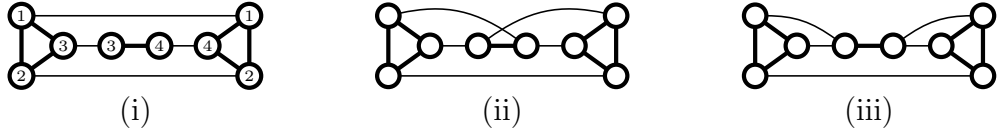

Suppose we are in (i) or (ii) of Figure~\ref{fig1} and $\beta$ uses a common
color, say 1, on two
vertices in the same clique $K$ in $S$.  We can recolor the third vertex of
$K$ (if it exists) to also use 1 (since 1 cannot be used on any other vertex). 
Now we are done, since $\chi(G-K)=2$, and thus $\Kc(G-K,3)=1$.  So it suffices
to prove that $\beta$ uses 1 on two vertices in the same clique in $S$.

In case (ii) this holds because
deleting the edges of $S$ and taking the complement yields a
graph with chromatic number 5.  In case (i), we have one possible 4-coloring
that uses no common color on endpoints of any edge in $S$.  However, 
two of its color classes together induce a 4-cycle in $\overline{G}$.  
So after one Kempe swap, we reach a coloring with a common color on endpoints of 
some edge in $S$, and we are done.  

Finally, we consider (iii).  The analysis is nearly the same as (ii).
The difference is that some edges of $S$ lie in copies of $K_3$ not in $S$.
So if a common color 1 is used in $\beta$ on a $K_3$, call it $K$, not in $S$,
we cannot necessarily move 1 to a clique of $S$.  However, each $K_3$ lies in a
spanning subgraph isomorphic to $S$.  Thus, from $\beta$ we can reach a
3-coloring, 
finishing as above.
\end{proof}

Theorem \ref{smallest-lem} motivates the following revised
version of Conjecture \ref{main-conj}.

\begin{conj}
\label{main-conj-revised}
If $G$ is a $(k-1)$-colorable $B + E_\ell$ graph with $k \ge 4$ and $\ell <
\frac{k^2+8k-45+\Ix{{k-1}}}4$, then $\Kc(G, k) = 1$.
\end{conj}

\subsection*{Acknowledgements}

Thanks to two referees for their constructive comments, which helped to improve our presentation.
Carl Feghali was supported by Agence Nationale de la Recherche (France) under
research grant ANR DIGRAPHS ANR-19-CE48- 0013-01.

\scriptsize{
\bibliographystyle{habbrv}
\bibliography{almost-bipartite}
}

\end{document}